\documentclass[11pt]{article}

\usepackage[T1]{fontenc}
\usepackage{lmodern}
\usepackage[british]{babel}
\usepackage[margin=1in]{geometry}
\usepackage{amsmath,amssymb,amsthm,amsfonts,mathtools}
\usepackage{microtype}
\usepackage{enumitem}
\usepackage{xcolor}
\usepackage[numbers,square,sort&compress]{natbib}
\usepackage{hyperref}

\hypersetup{
  colorlinks=true,
  linkcolor=blue!55!black,
  citecolor=blue!55!black,
  urlcolor=blue!55!black
}

\newtheorem{theorem}{Theorem}[section]
\newtheorem{proposition}[theorem]{Proposition}
\newtheorem{lemma}[theorem]{Lemma}
\newtheorem{corollary}[theorem]{Corollary}
\theoremstyle{definition}

\theoremstyle{remark}
\newtheorem{remark}[theorem]{Remark}

\newcommand{\Ric}{\operatorname{Ric}}
\newcommand{\Rm}{\operatorname{Rm}}
\newcommand{\Area}{\operatorname{Area}}
\newcommand{\Length}{\operatorname{Length}}

\newcommand{\p}{\partial}
\newcommand{\wt}{\widetilde}
\newcommand{\dmu}{\,d\mu}
\newcommand{\dsigma}{\,d\sigma}
\newcommand{\Rmax}{R_{\max}}

\title{A Variational Free-Boundary Ricci Flow on Surfaces with Boundary}
\author{%
  Jean Carlos Cortissoz\\
  {\small Department of Mathematics}\\
  {\small Universidad de los Andes}\\
  {\small Bogot\'a, Colombia}\\
  {\small\href{mailto:jcortiss@uniandes.edu.co}{jcortiss@uniandes.edu.co}}
  \and
  Juan Jos\'e Villamar\'in\\
  {\small Department of Mathematics}\\
  {\small University of Washington}\\
  {\small Seattle, WA}\\
  {\small\href{mailto:jjvc11@uw.edu}{jjvc11@uw.edu}}\\
  {\small Corresponding author}
}
\date{}

\begin{document}

\maketitle

\begin{abstract}
We study a Ricci flow on a compact surface with boundary arising as the $L^{2}$-gradient flow of the Liouville functional with its natural Neumann boundary condition.  Geometrically, this boundary condition preserves the geodesic-curvature density $H\dsigma$.  In the regime in which the
total scalar curvature satisfies $0<C_{0}<4\pi$, a scale invariant normalised Liouville
functional is monotone, and bounded below. In our first result, we use its
dissipation to show that every compact blow-up limit is homothetically shrinking with constant positive scalar curvature. We then treat a rotationally symmetric disc with $H_{0}>0$ and $C_{0}>0$, in which case every noncompact limit is ruled out via singularity analysis and a noncollapsing theorem. Consequently, the normalised
flow fully converges smoothly to the unique spherical cap determined by the
initial area and the preserved boundary data, yielding our second result. Additionally, we also study the regime in which $C_0$ is nonpositive. In this case standard PDE arguments prove full smooth convergence of the normalised flow towards a metric with constant nonpositive scalar curvature. 
\end{abstract}

\medskip




\section{Introduction}

Let $(M^{2},g_{0})$ be a compact orientable surface with smooth nonempty
boundary. Note that by compactness, the surface is automatically complete as a metric space. We consider the free boundary Ricci flow defined by
\begin{equation}\label{eq:free-boundary-flow-intro}
\begin{cases}
\p_{t}g=-R_{g}g & \text{in }M\times[0,T),\\
H_{g}\dsigma_{g}=H_{0}\dsigma_{0} & \text{on }\p M\times[0,T),\\
g(0)=g_{0}.
\end{cases}
\end{equation}
Here $H_{g}$ is computed with respect to the inward unit normal, with the
convention that the boundary of a Euclidean disc has positive geodesic
curvature.  If $g(t)=e^{2u(t)}g_{0}$, the boundary condition in
\eqref{eq:free-boundary-flow-intro} is precisely
$\p_{N_{0}}u=0$. As we will see, it is the natural boundary condition for the
Liouville functional without a boundary term.

Throughout this document we will use these three quantities:
\[
A_{0}:=\Area(M,g_{0}),
\qquad
C_{0}:=\int_{M}R_{0}\dmu_{0} \qquad B_0:= \int_{\p M}H_0 \dsigma_0.
\]
The boundary condition and the Gauss--Bonnet theorem imply that $C_{0}$ is preserved too.
The area of the unnormalised flow consequently satisfies
\[
A(t)=A_{0}-C_{0}t.
\]
When $C_{0}>0$, the maximal time is finite and at most $A_{0}/C_{0}$ and therefore, the flow becomes singular. The corresponding normalised version of the flow is
\begin{equation}\label{eq:normalisation-intro}
\wt g(\wt t)=\frac{A_{0}}{A(t)}g(t),
\qquad
\wt t(t)=\int_{0}^{t}\frac{A_{0}}{A(\tau)}\,d\tau.
\end{equation}
It has fixed area $A_{0}$ and, as expected, in the interior it evolves by
\[
\p_{\wt t}\wt g=-(R_{\wt g}-r)\wt g,
\quad \text{ where } \quad
r=\frac{C_{0}}{A_{0}}.
\]
Whereas the boundary condition is once again 
\[
\wt H d\wt \sigma = H_0 \dsigma_0.
\]
Our first result is a conditional compactness theorem. We would like to remark that it does not require
$R_{g_{0}}\geq0$. 
\begin{theorem}[Compact blow-up limits]\label{thm:compact-blow-up}
Assume \(
0<C_{0}<4\pi,
\) and let $g(t)$ be the maximal solution of
\eqref{eq:free-boundary-flow-intro} on $[0,T)$.  Suppose that there are times $t_{i}\uparrow T$, scales $Q_{i}\to\infty$, and basepoints $p_{i}\in M$
such that the rescaled metrics
\begin{equation}\label{eq:general-rescaling}
g_{i}(s)=Q_{i}g\left(t_{i}+\frac{s}{Q_{i}}\right)
\end{equation}
converge smoothly and globally, up to the boundary and on compact time
intervals, to a compact pointed limiting flow
$(M_{\infty},g_{\infty}(s),p_{\infty})$. Then every time slice of $g_{\infty}(s)$ has spatially constant positive
scalar curvature.  More precisely, for some $\rho>0$,
\[
R_{\infty}(s)=\frac{\rho}{1-\rho s},
\qquad
g_{\infty}(s)=(1-\rho s)g_{\infty}(0),
\qquad
\Area(M_{\infty},g_{\infty}(0))=\frac{C_{0}}{\rho}.
\]
If the rescaling is normalised by
$R_{g_{i}(0)}(p_{i})=1$, then $\rho=1$.

Furthermore,
\[
T=\frac{A_{0}}{C_{0}},
\]
and the normalised flow \eqref{eq:normalisation-intro} exists for every
$\wt t\geq0$, and the corresponding sequence
$\wt g(\wt t(t_{i}))$ converges, modulo the Cheeger--Gromov
diffeomorphisms, to a metric of constant scalar curvature $C_{0}/A_{0}$ and the limiting boundary density $H_\infty d\sigma_\infty$ can be identified with $H_0d\sigma_0$.
\end{theorem}
For a disc with $H_{0}>0$, Gauss--Bonnet makes the inequality $C_0<4\pi$ automatic; in that setting the only scalar-curvature hypothesis is $C_{0}>0$.
Our second theorem states that in the rotationally symmetric with convex boundary setting the compactness hypothesis in Theorem~\ref{thm:compact-blow-up} holds.  This gives the following
unconditional result.

\begin{theorem}[Rotationally symmetric disc]\label{thm:rotational-disc}
Let $(M,g_{0})$ be a smooth rotationally symmetric Riemannian disc.  Assume
\[
C_{0}=\int_{M}R_{0}\dmu_{0}>0,
\qquad
H_{0}>0.
\]
Then the normalised free-boundary Ricci flow exists for all
$\wt t\geq0$ and, after fixing the pole and the radial parametrisation,
converges in $C^{\infty}$ to the unique spherical cap
$g_{\mathrm{cap}}$ satisfying
\[
\Area(M,g_{\mathrm{cap}})=A_{0},
\qquad
R_{g_{\mathrm{cap}}}=\frac{C_{0}}{A_{0}},
\qquad
H_{g_{\mathrm{cap}}}\dsigma_{g_{\mathrm{cap}}}
=H_{0}\dsigma_{0}.
\]
\end{theorem}

Again, notice that no sign assumption is made on $R_0$, but this time, we do need $H_0$ to be positive. A rough idea of the proof is as follows. On one hand, the normalised
Liouville functional identifies compact blow-up limits as its critical points, while on the other hand Perelman's
entropy yields noncollapsing when $H\geq0$.  In the rotationally symmetric
case we rescale by
\[
Q_{i}=\max\{\Rmax(t_{i}),H(t_{i})^{2}\}.
\]
The genuinely $H$-dominated possibility, that is, 
\begin{equation*}
    \frac{ H^2(t_i)}{\Rmax (t_i)} \rightarrow \infty,
\end{equation*}
produces a flat limit with compact boundary, then
$H_{\infty}=1$; and the Riccati equation, as we will see, forces the limit to be compact. Then $C_0<4\pi$ leads to a contradiction with the Gauss--Bonnet theorem.  Once $H^{2}/\Rmax$ is controlled, any remaining noncompact limit either has no boundary or has totally geodesic boundary (and therefore can be
doubled).  The two-dimensional ancient-solution theory then produces a cigar as a further blow-up limit, contradicting non-local collapsing.\\

Dropping the symmetry hypothesis and assuming only that the boundary is convex we get the following dichotomy. 

\begin{theorem}\label{thm:nonsymmetric disc}
    Let \((M, g_0)\) be a smooth disc with convex boundary and positive total curvature \(C_0\). Let \( Q(t) = \max \{ R_{\max}(t) , H_{\max}(t)^2 \} \), then either:
    \begin{enumerate}
        \item there exists a sequence of times \(t_i \uparrow T\) such that \(Q_i = Q(t_i) = \max_{[0,t_i]}Q(t)\) and 
        \[
        \frac{H_{\max}(t_i)^2}{R_{\max}(t_i)} \rightarrow \infty.
        \]
        Then, choosing \(p_i\in \p M\) such that \(h_i = H(p_i, t_i) = Q_i\), the rescaled sequence \((M, g_i(s), p_i)\), where
    \[
        g_i(s) := h_i^2g(t_i + s/h_i^2)
    \]
    has a subsequence that converges in the Cheeger--Gromov sense to a noncompact flat manifold \(M_\infty\) with noncompact convex boundary \(\p M_\infty\); or 
    \item there is a constant \( C > 0\) such that for every sequence of times \(t_i \uparrow T\) such that \( Q_i = Q(t_i) = \max_{[0,t_i]} Q(t) \), then for every sufficiently large \( i\)
    \[
        H_{\max}(t_i)^2 \leq C R_{\max}(t_i).
    \]
    Then every blow-up limit is compact, \(T = T_{\text{area}}\) and normalised free-boundary Ricci flow exists for all \(\wt t\geq 0\). Every corresponding sequence \(\wt g (\wt t_i)\) has a subsequence converging smoothly in the Cheeger--Gromov sense to a spherical domain \(M_\infty\) with convex boundary satisfying 
    \[
    \Area(M_\infty) = A_0, \quad R_\infty = \frac{C_0}{A_0}, \quad H_\infty \dsigma_\infty = H_0 \dsigma_0.
    \]
    \end{enumerate}

\end{theorem}

Notice that when \( H_0 \equiv 0\) our flow reduces to the Ricci flow with totally geodesic boundary studied by S. Brendle in \cite{Brendle}, where exponential convergence towards a metric of constant scalar curvature and totally geodesic boundary is shown. However, the methods employed here do not prove convergence for the critical case \(C_0 = 4\pi\). Additionally, the authors would like to remark that a similar flow was considered in \cite{Li93}. In there the author imposes the boundary condition 
\[
\p_NR = 0
\]
along the flow, and proves long time existence of the normalised flow, and convergence to a smooth metric $g_{\infty}$ of constant curvature and geodesic curvature satisfying $H_\infty\dsigma_\infty = H_0 \dsigma_0$ in the nonpositive $C_0$ regime. As we will show, the boundary condition above is equivalent to our Neumann boundary condition for the conformal factor $u(t)$ or equivalently, to the geometric condition $H\dsigma = H_0\dsigma_0$. For the benefit of the reader we include in Section \ref{sec:nonpositive-total-curvature} a proof of this result adapted to our variational flow. As matter of fact, we prove the following.

\begin{theorem}[Non-positive total curvature]
\label{thm:nonpositive-curvature}
Assume that
\[
C_{0}=\int_{M}R_{g_{0}}\,d\mu_{0}\leq0,
\qquad
r=\frac{C_{0}}{A_{0}}.
\]
Then the normalised free-boundary Ricci flow
\[
\begin{cases}
\partial_{\wt t}\wt g
=-(R_{\wt g}-r)\wt g
& \text{in }M\times[0,\infty),\\
H_{\wt g}\,d\sigma_{\wt g}
=H_{0}\,d\sigma_{0}
& \text{on }\partial M\times[0,\infty),\\
\wt g(0)=g_{0}
\end{cases}
\]
has a unique smooth solution for all \(\wt t\geq0\).

Moreover, there exists a unique metric \(g_{\infty}\) in the conformal
class of \(g_{0}\) satisfying
\[
\Area(M,g_{\infty})=A_{0},
\qquad
R_{g_{\infty}}=\frac{C_{0}}{A_{0}},
\qquad
H_{g_{\infty}}\,d\sigma_{g_{\infty}}
=H_{0}\,d\sigma_{0},
\]
and
\[
\wt g(\wt t)\longrightarrow g_{\infty}
\qquad\text{in }C^{\infty}
\quad\text{as }\wt t\longrightarrow\infty.
\]
The convergence is exponential in every \(C^{k}\)-norm.  In particular,
\(g_{\infty}\) has constant negative scalar curvature when \(C_{0}<0\)
and is flat when \(C_{0}=0\).
\end{theorem}
\begin{remark}
No pointwise sign condition is required
on either \(R_{0}\) or \(H_{0}\).
\end{remark}

Finally, in the last section we relate our results to the asymptotic behaviour of the associated logarithmic diffusion equation.

\subsection{The Liouville functional and the free-boundary flow}

In this section we show that the Ricci flow \eqref{eq:free-boundary-flow-intro} corresponds to the negative \(L^2\) gradient flow of the Liouville functional introduced below.  Write $g=e^{2u}g_{0}$.  The conformal transformation laws are
\[
d\mu_{g}=e^{2u}d\mu_{0},
\qquad
d\sigma_{g}=e^{u}d\sigma_{0},
\]
\[
R_{g}=e^{-2u}(R_{0}-2\Delta_{0}u),
\qquad
H_{g}=e^{-u}(H_{0}+\p_{N_{0}}u).
\]
where  \( N_0\) represents the outward pointing unit normal. In particular, observe that
\[
H_{g}\dsigma_{g}=H_{0}\dsigma_{0}
\quad\Longleftrightarrow\quad
\p_{N_{0}}u=0.
\]

Next, we consider the Liouville functional, see \cite{Brendle}:
\begin{equation}\label{eq:liouville}
\mathcal L(u)=\int_{M}\bigl(|\nabla_{0}u|_{g_0}^{2}+R_{0}u\bigr)\dmu_{0}.
\end{equation}

\begin{proposition}\label{prop:first-variation}
For $u_{\varepsilon}=u+\varepsilon\varphi$,
\[
\left.\frac{d}{d\varepsilon}\right|_{\varepsilon=0}
\mathcal L(u_{\varepsilon})
=\int_{M}R_{g}\varphi\dmu_{g}
+2\int_{\p M}(\p_{N_{0}}u)\varphi\dsigma_{0}.
\]
Consequently, the natural boundary condition is
$\p_{N_{0}}u=0$, and with respect to the $L^{2}(d\mu_{g})$ metric the Ricci flow \eqref{eq:free-boundary-flow-intro} is precisely the negative gradient flow, up to the factor
$1/2$.
\end{proposition}

\begin{proof}
Integration by parts on the derivative of the Dirichlet energy term gives
\[
2\int_{M}\langle\nabla_{0}u,\nabla_{0}\varphi\rangle\dmu_{0}
=-2\int_{M}(\Delta_{0}u)\varphi\dmu_{0}
+2\int_{\p M}(\p_{N_{0}}u)\varphi\dsigma_{0}.
\]
Combining this with the variation of the $R_0$ term shows that the integrand of interior term is
$(R_{0}-2\Delta_{0}u)\varphi\,d\mu_{0}
=R_{g}\varphi\,d\mu_{g}$.
\end{proof}

Thus the unnormalised free-boundary flow is
\begin{equation}\label{eq:free-boundary-flow}
\begin{cases}
\p_{t}g=-R_{g}g & \text{in }M\times[0,T),\\
g(t)=e^{2u(t)}g_{0},\\
\p_{N_{0}}u=0 & \text{on }\p M\times[0,T),\\
g(0)=g_{0}.
\end{cases}
\end{equation}
Equivalently, the conformal factor solves the quasilinear Neumann problem
\[
\p_{t}u=e^{-2u}\left(\Delta_{0}u-\frac12R_{0}\right),
\qquad
\p_{N_{0}}u=0.
\]
The sign and the factor in this equation follow from
$2\p_{t}u=-R_{g}$.  Standard quasilinear parabolic theory gives short-time
existence and uniqueness, subject to the usual compatibility conditions at
$t=0$; see, for example, \cite{CortissozMurcia}.\\

The area fixing version is
\begin{equation}\label{eq:normalised-flow}
\begin{cases}
\p_{\wt t}\wt g=-(R_{\wt g}-r)\wt g,\\
\p_{N_{0}}\wt u=0,\\
\wt g(0)=g_{0},
\end{cases}
\qquad \text{ where }
r=\frac{C_{0}}{A_{0}}.
\end{equation}
Its conformal factor satisfies
\[
\p_{\wt t}\wt u
=e^{-2\wt u}\left(\Delta_{0}\wt u-\frac12R_{0}\right)
+\frac12r.
\]
\begin{remark}
    Note that under this flow and its normalised version the sign of the geodesic curvature is preserved, for 
    \[
    H = e^{-u}H_0.
    \]
    Additionally, \(\p_Nu = 0\) as long as the solution exists.
\end{remark}
\subsection{Evolution equations and finite-time singularities}

For a conformal variation $\p_{t}g=\psi g$, the boundary geodesic
curvature and line element evolve by
\[
H'=-\frac12\psi H+\frac12\p_{N}\psi,
\qquad
\p_{t}d\sigma=\frac12\psi\,d\sigma.
\]
For the unnormalised flow, preservation of $H\dsigma$ gives
\[
0 = (Hd\sigma)' = (H' - \frac{1}{2}HR)\, d\sigma.
\]
So that $H' = \frac{1}{2}HR$, substituting this into the first formula above yields that the scalar curvature also satisfies the Neumann condition
\begin{equation}\label{eq:curvature-neumann}
\p_{N}R=0\qquad\text{on }\p M.
\end{equation}
Similarly, the same calculation for the normalised flow gives
$H'=\frac{1}{2}(R-r)H$ and again $\p_{N}R=0$.

\begin{remark}\label{rmk: equivalence with Li}
For either the normalised or unnormalised free boundary Ricci flow, the
geodesic curvature and the boundary measure satisfy
\[
    \partial_t\!\left(H\,d\sigma\right)
    =
    -\frac12\,(\partial_N R)\dsigma.
\]
Consequently, along a smooth solution, the curvature
Neumann condition imposed in~\cite{Li93} is equivalent to the
pointwise preservation of \(H\,d\sigma\). In our formulation, however,
the latter arises first as the natural boundary condition for the
Liouville functional; hence \(\partial_N R=0\) is a consequence of the
variational formulation rather than an additional prescription.
\end{remark}

Differentiating the conformal transformation formula for the scalar-curvature gives the usual evolution equation for $R$ on the interior. Therefore, the complete evolution equations are
\begin{equation}\label{eq:curvature-evolution}
\p_{t}R=\Delta R+R^{2},
\qquad
\p_{\wt t}R=\Delta R+R(R-r),
\qquad
\p_{N}R=0.
\end{equation}
Because of this, we can apply the Hopf maximum principle and conclude that if
\[
m_{-}=\min\left\{0,\min_{M}R_{g_{0}}\right\},
\]
then
\begin{equation}\label{eq:negative-part-bound}
R(x,t)\geq\frac{m_{-}}{1-m_{-}t}.
\end{equation}
Thus the negative part of the scalar curvature stays bounded on every
finite time interval.

Since the boundary density is conserved, the Gauss-Bonnet theorem immediately gives conservation of total scalar
curvature:
\begin{equation}\label{eq:total-curvature}
\int_{M}R_{g(t)}\dmu_{g(t)}
=4\pi\chi(M)-2\int_{\p M}H_{0}\dsigma_{0}
=C_{0}.
\end{equation}
Consequently, using that 
\[
(\dmu)' = \frac{1}{2}\mathrm{tr}\Dot{g}\dmu
\]
we see that \(A'(t)=-C_{0}\) and
\begin{equation}\label{eq:area-linear}
A(t)=A_{0}-C_{0}t.
\end{equation}
If \(C_{0}>0\) , a smooth solution cannot exist beyond time
\[
T_{\mathrm{area}}=\frac{A_{0}}{C_{0}},
\]
so its maximal time \(T\) is finite and satisfies \(T\leq T_{\mathrm{area}}\).

A standard continuation argument shows that if $R$ stays bounded on a
finite interval, then so do $u$ and $H=e^{-u}H_{0}$; parabolic boundary
estimates then extend the solution, see for instance \cite{CortissozMurcia}.  We shall determine below, under the
hypotheses of our main theorems, that in fact \( T=T_{\mathrm{area}}\).

\subsection{The normalised Liouville functional}

For $A(u)=\int_{M}e^{2u}\dmu_{0}$ S. Brendle in  \cite{Brendle} used the functional
\begin{equation}\label{eq:normalised-liouville}
\mathcal F(u)
=\mathcal L(u)-\frac{C_{0}}{2}
\log\left(\frac{A(u)}{A_{0}}\right).
\end{equation}

\begin{remark}
    Note that in his case the flow has minimal boundary, hence, we have by the Gauss-Bonnet Theorem
    \[
    C_0 = 4\pi \chi(M).
    \]
\end{remark}

\begin{proposition}\label{prop:normalised-liouville}
The functional $\mathcal F$ has the following properties.
\begin{enumerate}[label=\textup{(\roman*)}]
\item It is invariant under constant rescalings of the metric:
$\mathcal F(u+c)=\mathcal F(u)$.

\item Its first variation is
\[
\delta\mathcal F(u)[\varphi]
=\int_{M}\left(R_{g}-\frac{C_{0}}{A(u)}\right)
\varphi\dmu_{g}
+2\int_{\p M}(\p_{N_{0}}u)\varphi\dsigma_{0}.
\]
Thus its critical points satisfy
\[
R_{g}=\frac{C_{0}}{A(u)},
\qquad
\p_{N_{0}}u=0.
\]

\item Along the unnormalised flow,
\begin{equation}\label{eq:liouville-dissipation}
\frac{d}{dt}\mathcal F(u(t))
=-\frac12\int_{M}(R-r_{g})^{2}\dmu_{g},
\qquad
r_{g}=\frac{C_{0}}{A(t)}.
\end{equation}
Along the fixed-area flow the same identity holds with $r_{g}=r$ and
$t$ replaced by $\wt t$.
\end{enumerate}
\end{proposition}

\begin{proof}
In order to see (i) we compute:
\begin{align*}
    \mathcal F(u+c) &= \int_M \left( |\nabla_0(u + c)|^2 + R_0(u+c) \right)d\mu_0 - \frac{C_0}{2}\log\left( \frac{1}{A_0}\int_M e^{2(u+c)}d\mu_0 \right)\\
    &= \mathcal{L}(u) + cC_0 - \frac{C_0}{2}\log\left( \frac{1}{A_0}\int_Me^{2c}d\mu_0\right) = \mathcal{F}(u)
\end{align*}
and scale invariance follows.  Proposition~\ref{prop:first-variation} and
$\delta A(u)[\varphi]=2\int_{M}\varphi\dmu_{g}$ give the first-variation
formula and (ii) holds true.
Along the unnormalised flow, $\p_{t}u=-R/2$, and therefore
\[
\frac{d}{dt}\mathcal F
=-\frac12\int_{M}(R-r_{g})R\dmu_{g}
=-\frac12\int_{M}(R-r_{g})^{2}\dmu_{g}.
\]
The last equality uses $\int_{M}(R-r_{g})\dmu_{g}=0$.  The normalised
identity is immediate.

\end{proof}

Besides the properties listed above, the normalised Liouville functional is bounded from below and coercive, as we will prove below. Before doing so, we recall the Moser-Trudinger inequality which will be employed on the proof.

\begin{theorem}[Moser-Trudinger inequality \cite{Yang}]\label{Th: MT inequality}
    Let $(M,g)$ be a compact 2-dimensional Riemannian manifold with boundary. Then the functional 
    \begin{equation*}
        \mathcal{MT}(v): = \int_M e^{2\pi v^2}d\mu
    \end{equation*}
    attains its supremum on the space 
    \begin{equation*}
        \mathcal{H}^1:= \{ u \in H^{1,2}(M) : \int_M|\nabla u|^2\,d\mu = 1, \int_Mu\,d\mu = 0\}.
    \end{equation*}
    Moreover, the function realising the supremum is smooth.
\end{theorem}

We will use this inequality to prove the following
\begin{corollary}
    For a smooth function $v$ defined on $M$ having zero mean the following inequality holds true:
    \begin{equation}\label{eq:log MT inequality}
        \log\left( \int_Me^{2v}\, d\mu \right) \leq \frac{1}{2\pi} \int_M |\nabla v|^2\, d\mu + C_*
    \end{equation}
\end{corollary}
\begin{proof}
    Let $v$ be such a function and define $w:= \frac{v}{\sqrt{E}}$ where $E$ is the Dirichlet energy of $v$. Thus, $w\in \mathcal{H}^1$. Observe that we have 
    \begin{equation*}
        2v = 2\sqrt{E}w \leq 2\pi w^2 + \frac{E}{2\pi}.
    \end{equation*}
    Therefore, taking the exponential on both sides and integrating we obtain 
    \begin{equation*}
        \int_Me^{2v}\,d\mu \leq e^{E/2\pi}\int_M e^{2\pi w^2}\, d\mu,
    \end{equation*}
    after applying the Moser-Trudinger inequality it yields
    \begin{equation*}
        \int_Me^{2v}\,d\mu \leq e^{E/2\pi}C_{MT},
    \end{equation*}
    where $C_{MT} = \sup \mathcal{MT}$. Finally the result follows after taking logarithm on both sides, note that the $C_*$ above is $\log C_{MT}$.
\end{proof}

\begin{proposition}
    The normalised Liouville functional $\mathcal{F}$ is bounded from below and moreover it is coercive provided $0< C_0< 4\pi$.
\end{proposition}
\begin{proof}
Let \(v=u-\bar u\) where 
\[
\bar u=\frac{1}{A_{0}}\int_{M}u\dmu_{0}.
\]
Then $\int_{M}v\dmu_{0}=0$ and scale invariance gives
\[
\mathcal F(u)
=\int_{M}|\nabla_{0}v|^{2}\dmu_{0}
+\int_{M}R_{0}v\dmu_{0}
-\frac{C_{0}}{2}
\log\left(\frac{1}{A_{0}}\int_{M}e^{2v}\dmu_{0}\right).
\]
The log Moser--Trudinger inequality \eqref{eq:log MT inequality} implies
\begin{equation}\label{eq:moser-trudinger}
\log\left(\frac{1}{A_{0}}\int_{M}e^{2v}\dmu_{0}\right)
\leq\frac{1}{2\pi}\int_{M}|\nabla_{0}v|^{2}\dmu_{0}+C,
\end{equation}
While the Poincaré and Cauchy--Schwarz inequalities give
\[
\left|\int_{M}R_{0}v\dmu_{0}\right|
\leq B\|\nabla_{0}v\|_{L^{2}}.
\]
For some constant \(B\). Hence, it follows that
\[
\mathcal F(u)
\geq\left(1-\frac{C_{0}}{4\pi}\right)
\|\nabla_{0}v\|_{L^{2}}^{2}
-B\|\nabla_{0}v\|_{L^{2}}-C.
\]
The leading coefficient is positive exactly when
$C_{0}<4\pi$, which proves both the lower bound and coercivity (modulo constants).
\end{proof}

It follows that, whenever $0<C_{0}<4\pi$ and the unnormalised flow has
maximal time $T$,
\begin{equation}\label{eq:liouville-limit}
\mathcal F_{T}:=\lim_{t\uparrow T}\mathcal F(u(t))
\end{equation}
exists and is finite. 
\section{Rigidity of compact blow-up limits}\label{sec: Rigidity of compact limits}

We now prove Theorem~\ref{thm:compact-blow-up}. Under the rescaling
\eqref{eq:general-rescaling}, the scalar curvature and volume form scalings cancel out. Indeed, recall that we have
\[
R_{g_{i}(s)}=Q_{i}^{-1}R_{g(t_{i}+s/Q_{i})},
\qquad
d\mu_{g_{i}(s)}=Q_{i}\,d\mu_{g(t_{i}+s/Q_{i})}.
\]
Hence, the total scalar curvature of the rescaled metrics is again $C_0$.
\[
\int_{M}R_{g_{i}(s)}\dmu_{g_{i}(s)}=C_{0}.
\]
Set
\[
A_{i}(s)=\Area(M,g_{i}(s)),
\qquad
r_{i}(s)=\frac{C_{0}}{A_{i}(s)}.
\]

\begin{proof}[Proof of Theorem~\ref{thm:compact-blow-up}]
Constant metric rescalings correspond to adding a constant to the conformal
factor.  Hence Proposition~\ref{prop:normalised-liouville} gives
\[
\mathcal F(g_{i}(s))
=\mathcal F\left(g\left(t_{i}+\frac{s}{Q_{i}}\right)\right)
\]
and
\[
\frac{d}{ds}\mathcal F(g_{i}(s))
=-\frac12\int_{M}(R_{g_{i}(s)}-r_{i}(s))^{2}
\dmu_{g_{i}(s)}.
\]
Fix $s_{1}<s_{2}$ in the time interval of the limiting flow.  Integration
gives
\begin{equation}\label{eq:rescaled-dissipation}
\frac12\int_{s_{1}}^{s_{2}}\int_{M}
(R_{g_{i}(s)}-r_{i}(s))^{2}\dmu_{g_{i}(s)}\,ds
=\mathcal F(g_{i}(s_{1}))-\mathcal F(g_{i}(s_{2})).
\end{equation}
For $j=1,2$,
\[
t_{i}+\frac{s_{j}}{Q_{i}}\longrightarrow T.
\]
Therefore, the right-hand side of \eqref{eq:rescaled-dissipation} tends to zero by \eqref{eq:liouville-limit}.\\

Assume the sequence \((M, g_i(s), p_i)\) converges smoothly and globally in the Cheeger--Gromov sense to the flow $(M_{\infty}, g_{\infty}(s), p_{\infty})$ where $M_{\infty}$ is compact with boundary. Note that by the compactness of $M_{\infty}$ we have that it is diffeomorphic to $M$. Let $\Phi_{i}\colon M_{\infty}\to M$ denote the global
Cheeger--Gromov diffeomorphisms.  Smooth global convergence gives by definition
\[
\Phi_{i}^{*}g_{i}(s)\longrightarrow g_{\infty}(s)
\]
on $M_{\infty}\times[s_{1},s_{2}]$, and in particular
\[
A_{i}(s)\longrightarrow A_{\infty}(s)
:=\Area(M_{\infty},g_{\infty}(s)).
\]
Passing to the limit in \eqref{eq:rescaled-dissipation} yields
\[
\int_{s_{1}}^{s_{2}}\int_{M_{\infty}}
\left(R_{\infty}-\frac{C_{0}}{A_{\infty}(s)}\right)^{2}
\dmu_{\infty}\,ds=0.
\]
Since the integrand is nonnegative we obtain that
\[
R_{\infty}(s)\equiv\frac{C_{0}}{A_{\infty}(s)}
\]
for every $s$.  In particular, if
$\rho=R_{\infty}(0)$, then $\rho>0$ and $M_{\infty}$ has positive scalar curvature for all $s$.

The limiting scalar curvature satisfies the same evolution equation
\[
\p_{s}R_{\infty}=\Delta_{\infty}R_{\infty}
+R_{\infty}^{2}.
\]
But, since $R_\infty$ is spatially constant, the equation above reduces to $R_{\infty}'=R^{2}_{\infty}$. Solving this ODE we get that
\[
R_{\infty}(s)=\frac{\rho}{1-\rho s},
\qquad
g_{\infty}(s)=(1-\rho s)g_{\infty}(0).
\]
Since the total scalar curvature of the compact limit is $C_{0}$,
\[
A_{\infty}(0)=\frac{C_{0}}{\rho}.
\]

Now we show that the maximal time of existence coincides with $T_{\text{area}}$. To this end, note that
\[
 Q_{i}A(t_{i}) = A_{i}(0)\longrightarrow A_{\infty}(0)< \infty.
\]
Because $Q_{i}\to\infty$, it follows that $A(t_{i})\to0$.  From
\eqref{eq:area-linear},
\[
T=\frac{A_{0}}{C_{0}},
\qquad
A(t)=C_{0}(T-t).
\]
Indeed, if $T< T_{\text{area}}$ held, then we would have 
\[
0 = \lim_{ i \rightarrow \infty} A(t_i) = A_0 - C_0T > 0.
\]
Therefore, the desired equality $T= T_{\text{area}}$ holds and the maximal normalised time is
\[
\wt t(t)=\frac{A_{0}}{C_{0}}
\log\left(\frac{T}{T-t}\right)\longrightarrow\infty.
\]
Finally, regarding the normalised flow we have 
\[
\wt g(\wt t_i) = \frac{A_0}{A(t_i)}g(t_i) = \frac{A_0}{Q_iA(t_i)}g_i(0) \longrightarrow \frac{A_0}{A_{\infty}(0)}g_{\infty}(0).
\]
Where the convergence above is to be understood modulo the diffeomorphisms $\Phi_i$ from the Cheeger-Gromov convergence introduced above. Furthermore, the scalar curvature of the last metric is
\[
\wt R_{\infty} = \frac{A_{\infty}(0)}{A_{0}}\rho=\frac{C_{0}}{A_{0}} = r_0.
\]
Since $H\dsigma$ is unchanged by a constant rescaling, its pull-back also
passes to the smooth compact limit.  This proves the theorem.
\end{proof}

\begin{remark}
Notice that there is no topological restriction. To see this, note that the only analytical restriction, $0< C_0< 4\pi$, gives 
\[
1 - 2g - b < \frac{1}{2\pi}\int_{\p M} H_0\,d\sigma_0 < 2 - 2g - b.
\]
That is, for every choice of genus $g$ and number of boundary components $b$ the integral of $H_0$ lies on a nonempty interval. Hence, on any compact surface we can conformally deform the metric so that the integral of $H_0$ lies on that interval.
\end{remark}

\section{Perelman's entropy and noncollapsing with boundary}

This section is devoted to prove that under our variational flow the usual Perelman entropy functional is monotonically nondecreasing with the further assumption of the convexity of the boundary, that is, $H_0\geq 0$. This will imply a noncollapsing theorem and give the needed injectivity radius estimates for our blow-up limits. Let
\[
U=(4\pi\tau)^{-1}e^{-f},
\qquad
\int_{M}U\dmu=1,
\]
and take the Perelman's entropy functional in dimension two, see \cite{ChowLuNi,Perelman}
\begin{equation}\label{eq:perelman-W}
\mathcal W(g,f,\tau)
=\int_{M}\left[\tau(R+|\nabla f|^{2})+f-2\right]U\dmu.
\end{equation}
The letter $U$ is used here to avoid confusion with the conformal factor
$u$.

\begin{proposition}[Entropy monotonicity]\label{prop:perelman-monotonicity}
Suppose $g(t)$ solves the unnormalised free-boundary flow,
$\tau'=-1$, and $f$ solves the usual conjugate heat equation on the interior and satisfies the boundary condition
$\p_{N}f=0$.  Then
\begin{equation}\label{eq:perelman-monotonicity}
\frac{d}{dt}\mathcal W
=2\tau\int_{M}
\left|\Ric+\nabla^{2}f-\frac{1}{2\tau}g\right|^{2}U\dmu
+2\tau\int_{\p M}H|\widehat\nabla f|^{2}U\dsigma.
\end{equation}
In particular, $\mathcal W$ is nondecreasing whenever $H\geq0$.
\end{proposition}

\begin{proof}
The boundary variation formula of Ecker and Cortissoz--Murcia
\cite{Ecker,CortissozMurcia} is
\[
\frac{d}{dt}\mathcal W
=2\tau\int_{M}
\left|\Ric+\nabla^{2}f-\frac{1}{2\tau}g\right|^{2}U\dmu
+\tau\int_{\p M}
\bigl(RH-2\p_{t}H+2H|\widehat\nabla f|^{2}\bigr)U\dsigma.
\]
For our boundary condition,
$RH-2\p_{t}H=\p_{N}R=0$ by
\eqref{eq:curvature-neumann}, which gives
\eqref{eq:perelman-monotonicity}.
\end{proof}

Let $\mu(g,\tau)$ be, as usual, the infimum of $\mathcal{W}$ over
normalised $H^{1}$ functions; the minimiser satisfies the natural Neumann
condition.  Proposition~\ref{prop:perelman-monotonicity} gives the usual
monotonicity of $\mu$ along the free-boundary flow when $H_0\geq0$.

\begin{theorem}[No local collapsing]\label{thm:no-local-collapsing}
Let $(M,g(t))$, $t\in[0,T)$, be a maximal free-boundary Ricci flow on a
compact surface, where $T<\infty$ and $H_{0}\geq0$.  For every
$\rho\in (0, \infty)$ there is $\kappa=\kappa(g_{0},T,\rho)>0$ such that the flow is
$\kappa$-noncollapsed below scale $\rho$,  more precisely, if
$0<r\leq\rho$ and
\[
|R_{g(t)}|\leq r^{-2}
\quad\text{on }B_{g(t)}(x,r),
\]
then
\begin{equation}\label{eq:kappa-noncollapsing}
\Area_{g(t)}(B_{g(t)}(x,r))\geq\kappa r^{2}.
\end{equation}
The ball may meet the boundary.
\end{theorem}

\begin{proof}
Applying the logarithmic Sobolev inequality on the compact initial surface gives a
uniform lower bound for $\mathcal{W}$: 
\[
\mathcal{W} \geq -C,
\]
thus, we conclude that $\mu(g_{0},\tau) \geq -C$ for
$0<\tau\leq T+\rho^{2}$.  Entropy monotonicity therefore implies
\[
\mu(g(t),r^{2})\geq\mu(g_{0},t+r^{2})\geq-C
\]
for $t<T$ and $r\leq\rho$.

In order to finish the proof, we will show that if $(M, g(t))$ were $\kappa$-collapsed at the scale of $r$ at time a $t_0\geq T/2$, then $\mu$ would be very large and negative. Suppose that a ball satisfying the curvature hypothesis has
$\Area B(x,r)<\kappa r^{2}$.  Let $\Phi:= e^{-f/2}$, replacing this into the entropy functional we obtain the following
\[ 
\mathcal{W}(g,f,r^2) = \frac{1}{4\pi r^2}\int_M[(r^2R - 2\log \Phi - 2)\Phi^2 + 4r^2|\nabla \Phi|^2]d\mu.
\]
As suggested by Perelman in \cite{Perelman}, introduce the standard cut-off
\[
\Phi(y)=e^{-c/2}\phi\left(\frac{d_{g(t)}(x,y)}{r}\right),
\]
where $\phi:\mathbb{R} ^{\geq0} \rightarrow \mathbb{R} $ is a smooth function equal to $1$ on $[0,1/2]$, vanishes on $[1,\infty)$ and is monotonically decreasing to $0$ on $[1/2,1]$, and $c$
normalises $(4\pi r^{2})^{-1}\int\Phi^{2}\dmu$ to \(1\).  The curvature and
gradient terms in $\mathcal W$ are uniformly bounded, while the
normalising constant contributes the logarithm of the area ratio.  The
standard calculation from Perelman's proof \cite{Perelman,ChowLuNi} gives
\[
\mu(g(t),r^{2})\leq \log\kappa+C'.
\]
This contradicts the preceding lower bound once $\kappa$ is sufficiently
small.  Because \eqref{eq:perelman-W} has no boundary integral, the same
test function works for a ball meeting $\p M$.
\end{proof}

The noncollapsing estimate is invariant under parabolic rescaling and
passes to pointed smooth limits and to further blow-ups.  The cigar soliton
is collapsed at large scales.  We shall use this incompatibility to exclude
noncompact ancient limits. Additionally, following the arguments appearing in \cite{ChowLuNi} we can prove the following injectivity radius estimate.

\begin{remark}
    For a manifold with boundary we define the injectivity radius of a boundary point \(p\) as the supremum \(r>0\) such that any geodesic starting at \(p\) normal to \( \p M\) uniquely minimises distance to \( \p M\) from \(0\) to \(r\). Additionally, for an interior point, the injectivity radius is defined as the supremum \( r\) such that any geodesic issuing from \(p\) minimises arclength from \( 0\) to \( \min\{ t_\gamma, r \}\) where \(t_\gamma\) is the first time \( \gamma \) hits the boundary.
\end{remark}
\begin{proposition}\label{prop: injectivity radius}
    Let $(M,g(t))$ be a maximal solution to the Ricci flow \eqref{eq:free-boundary-flow-intro} on $[0,T)$ with $T<\infty$ on a compact orientable surface with boundary. For every constant $C<\infty$ there exists a constant $a>0$ depending only on $C, g_0$ and $T$ such that if $(x,t)$ is a point and time such that on the ball $B_{g(t)}(x, 1/\sqrt{CK}\,)$ the bound
    \[
    |R| \leq CK,
    \]
    holds, and if the ball meets the boundary, 
    \[
    |H|\leq \sqrt{CK}
    \]
    holds too, where $K:= |R(x,t)|>0$, then 
    \begin{equation}
        \mathrm{inj}_{g(t)}(x) \geq \frac{a}{\sqrt{K}}.
    \end{equation}
\end{proposition}
 In order to prove this we begin by showing a Cheeger-Gromov-Taylor theorem for manifolds with boundary. We refer the reader to \cite{ChowLuNi} for its closed case.
 \begin{theorem}[Cheeger-Gromov-Taylor Theorem]\label{Th: CGT}
     For any constants $c>0, \, r_0>0$ there exists a constant $\iota_0>0$ such that if $(M^n,g)$ is a complete Riemannian manifold with boundary satisfying $|K|\leq 1$ and $|A|\leq 1$ (where $K$ represents the sectional curvature and $A$ the second fundamental form of the boundary) and if $p$ is a point with 
     \[
     \frac{\mathrm{Vol}(B(p,r))}{r^n}\geq c
     \]
     for all $r\leq r_0$ then 
     \[
     \mathrm{inj}(p) \geq \iota_0.
     \]
 \end{theorem}
 \begin{proof}
         If this were false, take a sequence of pointed manifolds $(M_i,g_i,p_i)$ satisfying the curvature, second fundamental form and volume bounds but whose injectivity radii at $p_i$ satisfy
         \[
         \delta_i = \mathrm{inj}(p_i) \longrightarrow 0.
         \]
         Consider then the rescaled metrics $\displaystyle \frac{1}{\delta_i^2}g_i$. Then, after the rescaling we get 
         \[
         |K_i| \leq \delta_i^2, \quad |A_i|\leq \delta_i, \quad \mathrm{inj}(p_i, \frac{1}{\delta_i^2}g_i) = 1.
         \]
         So that, we can take a Cheeger-Gromov limit $(M_{\infty}, g_{\infty}, p_{\infty})$ which will be complete and noncompact and possibly with boundary. To see that it is noncompact, note that the volume of the balls around $p_i$ tends to infinity. Furthermore, if the limit has nonempty boundary $\p M_{\infty}$ then it will be totally geodesic and it can be doubled to obtain a noncompact boundaryless manifold. Additionally, notice that $M_{\infty}$ will be a flat manifold. Nevertheless, every noncompact complete flat manifold is either (globally) isometric to $\mathbb{R}^n$ or otherwise satisfies 
         \[
         \lim_{r\rightarrow \infty} \frac{\mathrm{Vol}B(p,r)}{r^n} = 0.
         \]
         The first case is impossible because in $M_{\infty}$ the injectivity radius at $p_{\infty}$ is bounded above by 1. While the second case is precluded by the volume bound hypothesis. Indeed, we have 
         \[
         \mathrm{Vol}_{g_{\infty}}(B(p_{\infty}, r)) \geq cr^n
         \]
         for all $r\in [0,\infty)$.
     \end{proof}
     \begin{proof}[Proof of Proposition \ref{prop: injectivity radius}] By the non-local collapsing theorem \ref{thm:no-local-collapsing} we have that if
     \[
     |R|< CK,
     \]
     then there exists a $\kappa>0$ such that the solution is $\kappa$-noncollapsed at the scale of $1/\sqrt{CK}$. Then, after rescaling the metric by \(CK \), Theorem \ref{Th: CGT} gives the desired bound on the injectivity radius.
         
     \end{proof}
\section{The rotationally symmetric disc}

We now prove Theorem~\ref{thm:rotational-disc}.  Let us first quickly show that rotational symmetry is
preserved under the flow. 
\begin{lemma}
    Let $X$ be any infinitesimal symmetry of $g_0$. Then, $X$ is also an infinitesimal symmetry of the solution $g(t)$ of the free boundary Ricci flow \eqref{eq:free-boundary-flow-intro}.
\end{lemma}
\begin{proof}
    Let $\phi_s: M \rightarrow  M$ be the global flow generated by $X$. Set $\bar g(t):= \phi_s^*g(t)$, and let us show $\bar g$ is also a solution with the same initial condition. Indeed, we have 
    \[
    \p_t\bar g = \p_t(\phi_s^* g) = \phi_s^*(-Rg) = -\bar R \bar g,
    \]
    and similarly for the geodesic curvature density 
    \[
    \bar H \,d\bar \sigma = \phi_s^*(H\dsigma) = \phi_s^*(H_0\dsigma_0) = H_0\dsigma_0,
    \]
    and of course, $\bar g(0) = \phi_s^*g(0) = g_0$. Therefore, by the uniqueness of the solution, we conclude
    \[
    \phi_s^*g = \bar g = g.
    \]
\end{proof}
From now on, let $(M,g_0)$ be a rotationally symmetric disc with positive total curvature $C_0$ and positive geodesic curvature $H_0$. Thus, the geodesic curvature $H(t)$ is constant in space on the boundary for all $t \in [0,T)$.  Let
\[
\ell(t)=\Length_{g(t)}(\p M),
\qquad
B_{0}=\int_{\p M}H_{0}\dsigma_{0}.
\]
Preservation of the boundary density gives
\begin{equation}\label{eq:boundary-length-product}
H(t)\ell(t)=B_{0}.
\end{equation}
Since $M$ is a disc, the Gauss--Bonnet theorem gives
\begin{equation}\label{eq:B0}
B_{0}=2\pi-\frac{C_{0}}{2}.
\end{equation}
The assumptions $H_{0}>0$ and $C_{0}>0$ therefore imply
\[
0<B_{0}<2\pi,
\qquad
0<C_{0}<4\pi.
\]

\subsection{Blow-up limits under the combined scale}

Recall that if the maximal time $T< \infty$, then
\[
\sup_{0\leq t \leq T}\left( \sup_M |R| + \sup_{\p M} |H|^2 \right) = \infty,
\]
see \cite{GianniotisII} for more details. Hence, we introduce the natural scale
\[
\mathcal Q(t)=\max\{\Rmax(t),H(t)^{2}\}.
\]
Choose times $t_{i}\uparrow T$ such that
\[
Q_{i}=\mathcal Q(t_{i})
=\max_{0\leq t\leq t_{i}}\mathcal Q(t)
\longrightarrow\infty,
\]
and consider the scaled metrics
\begin{equation}\label{eq:combined-rescaling}
g_{i}(s)=Q_{i}g\left(t_{i}+\frac{s}{Q_{i}}\right),
\qquad
-Q_{i}t_{i}\leq s\leq Q_i(T - t_i).
\end{equation}
On every compact time interval contained in \((-Q_it_i, 0] \), we have that
\[
R_{i}\leq1,
\qquad
0<H_{i}\leq1.
\]
The lower bound \eqref{eq:negative-part-bound} gives
\[
\liminf_{i\to\infty}R_{i}\geq0.
\]
Indeed, after rescaling we get 
\[
R_i(s) \geq \frac{1}{Q_i}\frac{m_-}{1 - m_-(t_i + s/Q_i)} \geq \frac{m_-}{Q_i} \rightarrow 0.
\]
Thus every smooth limit has nonnegative scalar curvature even though
$R_{0}$ may change sign.

We record the boundary length estimates required for compactness.  At a fixed
time write a rotational metric, using distance from the pole, as
\[
g_{i}=d\rho^{2}+\varphi_{i}(\rho)^{2}d\theta^{2},
\qquad
0\leq\rho\leq L_{i},
\]
where
\[
\varphi_{i}(0)=0,
\qquad
\varphi_{i}'(0)=1.
\]
The number $L_{i}$ is the boundary injectivity radius as defined above, then recall that
\[
H_{i}=\frac{\varphi_{i}'(L_{i})}{\varphi_{i}(L_{i})}.
\]

\begin{lemma}[Rotational boundary collar]\label{lem:boundary-collar}
Suppose $|R_{i}|\leq C_{R}$ and $0<H_{i}\leq C_{H}$.  Then
\[
L_{i}\geq L_{0}>0,
\]
where $L_{0}$ depends only on $C_{R}$ and $C_{H}$.
\end{lemma}

\begin{proof}
Recall that for distance spheres we have that 
\[
H(r)=\frac{1}{r}+O(r),
\]
see \cite{JohnLee} for a proof of this. In our case, the boundary is precisely the distance sphere of radius $L_i$, hence, we obtain
\[
H_i = \frac{1}{L_i} + O(L_i).
\]
Therefore, if there were no such $L_0$ we would have $L_i\rightarrow 0$. By the equation above, 
\[
H_i \rightarrow \infty
\]
which contradicts the hypothesis.
\end{proof}
\begin{remark}
    The Lemma above together with the well known Klingenberg's injectivity radius estimate from comparison geometry give another proof, in the rotationally symmetric case, for the injectivity radius estimate needed for the compactness theorem.
\end{remark}
\begin{corollary}
    In the combined scale, \eqref{eq:boundary-length-product} gives \( H_{i}\ell_{i}=B_{0}\), so that we have the uniform bound
\[
\ell_{i}\geq B_{0}.
\]
\end{corollary}

The remaining boundary hypotheses in compactness theorem in \cite{CortissozMurcia} are
also uniform.  Indeed, rotational symmetry makes both $R_{i}|_{\p M}$ and
$H_{i}$ spatially constant, and
\[
\p_{s}H_{i}=\frac12R_{i}|_{\p M}H_{i},
\qquad
\p_{s}d\sigma_{i}=-\frac12R_{i}|_{\p M}d\sigma_{i}.
\]
Thus $\p_{s}(H_{i}\dsigma_{i})=0$, and the boundary metric changes only by
a spatially constant factor.  Boundary derivative estimates
\cite{Gianniotis,CortissozMurcia} control the higher time derivatives on
compact intervals away from the initial rescaled time.

We may therefore apply the said compactness theorem in  \cite{CortissozMurcia}.  After passing to a subsequence, \eqref{eq:combined-rescaling} converges smoothly in the pointed Cheeger--Gromov sense to a complete ancient rotationally symmetric flow
\[
(M_{\infty},g_{\infty}(s),p_{\infty}),
\qquad
s\in(-\infty,0].
\]
At this stage the limit is not assumed to be compact. Nevertheless, let us now prove that under the hypothesis of positive scalar curvature and compact convex boundary a surface with compact boundary must be compact.

\begin{proposition}[Focusing Lemma]\label{prop: Focusing Lemma}
    Let $(\Sigma, g)$ be a complete surface with compact boundary such that $R \geq 0$ and the geodesic curvature $H\geq h_0 >0$. Then $\Sigma$ is compact.
\end{proposition}
\begin{proof}
    We will prove that under the hypotheses no point of $\Sigma$ has distance to the boundary greater than \( 1/h_0\), where \(h_0 >0\) is a lower bound for the geodesic curvature $H$. Then $\Sigma$ has finite diameter, for 
    \[
    \mathrm{diam} \,\Sigma \leq \frac{2}{h_0}+ \mathrm{diam} \, \p \Sigma.
    \]
    This together with the completeness imply \(\Sigma\) must be compact. To this end, let \(x\in \Sigma\) be an arbitrary point. Since \(\Sigma\) is complete, there exists a minimising geodesic \(\gamma\) from a point \(p\in \p \Sigma\) to  \(x\) whose length is equal to \(d(x,\p \Sigma)\). Then, for any normal vector field \(V\) along \(\gamma\) we have that the index form satisfies
    \[
    I(V,V) = \int_0^d \left( |D_tV|^2 - \Rm(V, \Dot{\gamma}, \Dot{\gamma}, V) \right) \, dt  + \langle \nabla_VV, \Dot{\gamma}\rangle \Big|_{t=0}^{t= d}\geq 0.
    \]
    Any such $V$ can be expressed as \( V(t) = v(t) E(t) \), where \(E\) is the parallel transport along \(\gamma\) of the unit tangent vector \( T \) to \(\p \Sigma\) at \(p\). For an admissible variation \( V(d) = 0\). Taking, for example $v(t) = 1 -t/d$, we get that 
    \[
    \int_0^d \left( \frac{1}{d^2} - \frac{R}{2}(1 - \frac{t}{d})^2 \right)\, dt  - H^{\p \Sigma}(p) \geq 0.
    \]
    Integrating and using that $R \geq 0$, we have that 
    \[
        \frac{1}{d} - H(p) \geq 0.
    \]
    Therefore, this gives the following bound for any \( x \in \Sigma\): 
    \[
    d(x, \p \Sigma) \leq \frac{1}{H(p_x)} \leq \frac{1}{h_0},
    \]
    as we desired to prove.

    Equivalently,  for any Jacobi field $J(t) = j(t)E(t)$ along \(\gamma\) the Jacobi equation in dimension two yields
    \[
    j''(t) + \frac{R}{2}j(t) = 0.
    \]
    So that $j''(t) \leq 0$. Additionally, recall that 
    \[
    j'(0) E(0) = D_tJ(0) = \nabla_{\Dot{\gamma}}J(0) = j(0)\nabla_T \Dot{\gamma} = -j(0)H(0)E(0).
    \]
    That is, \(j'(0) = -H(0)j(0)\). Applying Taylor's theorem we have 
    \[
    j(t) \leq j(0) (1 - H(0)t) \leq j(0) (1 - h_0t),
    \]
    which implies that any normal geodesic issuing from \( \p \Sigma\) will have a conjugate point at a distance at most \( 1/h_0\) and consequently, any such geodesic stops minimising arclength before that distance.
\end{proof}

\begin{remark}
The argument above is the two-dimensional Riemannian form of the Raychaudhuri equation.  If the expansion
$\Theta=- H\leq - h_0$, then
\[
\Theta'=-K-\Theta^{2}.
\]
The initial negative expansion forces focusing of the congruence of geodesics normal to the boundary at time at most $1/h_0$ .
\end{remark}

\subsection{Excluding the genuinely boundary-dominated scale}

In this subsection we show that the quantity
\[
\frac{H^2}{\Rmax}
\]
does not blow-up along the flow and consequently, we can utilise the $R$ scale to take the blow-up limit as it is customary.
\begin{proposition}\label{prop:H-over-R-bound}
There is a constant $C<\infty$ such that, for all $t$ sufficiently close
to $T$,
\begin{equation}\label{eq:H-over-R-bound}
H(t)^{2}\leq C\Rmax(t).
\end{equation}
\end{proposition}

\begin{proof}
Suppose instead that, for a sequence of times \(t_i \uparrow T\) we have 
\[
\frac{\Rmax(t_{i})}{H(t_{i})^{2}}\longrightarrow0.
\]
Then, rescale by $Q_{i}=H(t_{i})^{2}$ and take basepoints $p_i$ to be in the boundary.  Then
\[
H_{g_{i}(0)}=1,
\qquad
\ell_{i}(0)=H(t_{i})\ell(t_{i})=B_{0}.
\]
Moreover,
\[
\frac{m_{-}}{Q_{i}(1-m_{-}t_{i})}
\leq R_{g_{i}(0)}
\leq\frac{\Rmax(t_{i})}{H(t_{i})^{2}}.
\]
Thus the limiting manifold is flat, has a compact circular boundary of
length $B_{0}$, and has geodesic curvature $H_{\infty}=1$. Therefore, the Focusing Lemma \ref{prop: Focusing Lemma} implies the limit manifold is compact. Applying the Gauss Bonnet theorem yields
\[
\ell_{\infty} = H_{\infty}\ell_{\infty} = \int_{\p M_{\infty}}H_{\infty}\dsigma_{\infty}= 2\pi
\]
However, on the other side we have
\[
\ell_{\infty} = H_{\infty}\ell_{\infty} = \int_{\p M}H_0\dsigma_0 = - \frac{C_0}{2} + 2\pi < 2\pi.
\]
Which is impossible, showing the desired result.
\end{proof}

\subsection{Excluding noncompact curvature-scale limits}

We now rescale using $Q_i$'s as the scalar curvature:
\begin{equation}\label{eq:R-rescaling}
g_{i}(s)=\Rmax(t_{i})
g\left(t_{i}+\frac{s}{\Rmax(t_{i})}\right),
\qquad
R_{g_{i}(0)}(p_{i})=1.
\end{equation}
Proposition~\ref{prop:H-over-R-bound} in the preceding section controls the rescaled geodesic curvature of the boundary. The compactness theorem above therefore gives a complete ancient rotationally symmetric limit with
\[
R_{\infty}\geq0,
\qquad
R_{\infty}(p_{\infty},0)=1.
\]
In particular, the limit is nonflat.

The Focusing Lemma \ref{prop: Focusing Lemma} precludes the case in which \(M_\infty\) is noncompact but \(\p M_\infty\) is compact, since on this case we have 
\[
H_{\infty}\ell_{\infty} = B_0 >0,
\]
so that \(H_{\infty}>0\) as well. Hence, there are two remaining noncompact configurations.
\begin{enumerate}[label=\textup{(\roman*)}]
\item The boundary escapes to infinite distance from $p_{i}$, so the limit is boundaryless.

\item The boundary stays within finite distance from the centre. Thus, the boundary limit is noncompact. Hence, the rescaled lengths of the boundary tend to infinity and as a consequence
\[
H_{i}=\frac{B_{0}}{\ell_{i}}\longrightarrow0,
\]
so the limiting boundary is totally geodesic.  Reflecting across it gives a smooth complete ancient boundaryless flow; the smoothness of this doubling is proved in \cite[Appendix B]{CortissozMurcia}.
\end{enumerate}
In either case we obtain a complete nonflat noncompact ancient surface with bounded nonnegative curvature.  The strong maximum principle gives positive curvature.  By the two-dimensional ancient-solution results in \cite[Theorem 9.3 and Propositions 9.16 and 9.18]{ChowLuNi}, a backward Type-II blow-up produces the cigar soliton.  The $\kappa$-noncollapsing from Theorem~\ref{thm:no-local-collapsing} passes first to the ancient limit and then to this further blow-up.  This is impossible because the cigar is collapsed at all large enough scales.

\begin{remark}
    Moreover, recall that the total scalar curvature of the cigar is exactly \( 4\pi\). However, further assuming that \(R_0\) is nonnegative, take \( K \subseteq M_\infty\) be any compact set and let \( K_i\) be its image under the Cheeger-Gromov diffeomorphisms. Then, we have that 
    \begin{align*}
        \int_{K}R_{\infty}\dmu_{\infty} &= \lim_{i\rightarrow \infty}\int_{K_i}R_i\dmu_i \leq C_0 < 4\pi
    \end{align*}
    Exhausting \( M_\infty\) with compact sets gives that its total curvature is at most \( C_0 \).
    Hence, this gives yet another proof, without collapsing, that the cigar cannot be \textit{a boundaryless} blow-up limit of the Ricci flow \eqref{eq:free-boundary-flow-intro} with nonnegative initial scalar curvature and the rescalings given by \eqref{eq:R-rescaling}. 
\end{remark}

Thus every pointed limit is compact.  In the rotational disc setting such a compact limit contains both the pole and the boundary circle. We have proved the following.

\begin{proposition}\label{prop:all-R-limits-compact}
Every curvature-scale blow-up of the rotationally symmetric
flow has a compact limit, and the convergence to that limit is global not just local.
\end{proposition}

\subsection{Convergence of the normalised flow}

In this subsection we show that the limit is unique modulo diffeomorphisms, since the spherical cap obtained is completely determined by the initial data of the flow.

Indeed, a rotationally symmetric constant-positive-curvature metric on a smooth
disc with constant boundary geodesic curvature is a spherical cap.  It is uniquely determined by the area and the preserved boundary data.  To see this
explicitly, let $a$ be the radius of the ambient round sphere and let
$\vartheta$ be the polar angle of the cap.  Then
\[
a=\sqrt{\frac{2A_{0}}{C_{0}}},
\qquad
\cos\vartheta=\frac{B_{0}}{2\pi}
=1-\frac{C_{0}}{4\pi}.
\]
The second equation comes from writing the metric in the form 
\[
g_{\infty} = a^2(d\phi^2 + \sin^2(\phi)d\theta^2), \quad \text{for } \phi \in [0, \vartheta].
\]
Then, recall the well known formulas
\[
H_{\infty} = \frac{1}{a}\cot(\vartheta) \quad \text{and } \ell_{\infty} = 2\pi a \sin \vartheta.
\]
These equations determine a unique cap with
$0<\vartheta<\pi/2$.  Thus every convergent subsequence has the same limit.
After fixing the pole and the radial parametrisation, rotational symmetry
removes the Cheeger--Gromov diffeomorphism ambiguity.  Sequential
precompactness and uniqueness of the subsequential limit therefore imply
\[
\wt g(\wt t)\longrightarrow g_{\mathrm{cap}}
\qquad\text{in }C^{\infty}\text{ as }\wt t\to\infty.
\]
This completes the proof of Theorem~\ref{thm:rotational-disc}.

\section{Nonsymmetric disc}
Throughout this section let \((M, g_0)\) be a disc with convex boundary and whose total curvature \(C_0\) is positive. Let \( Q(t)\) be defined by \(\max\{ R_{\max}(t), H_{\max}(t)^2 \}\) and choose times \(t_i \uparrow T\) such that \( Q_i:= Q(t_i) = \max_{[0,t_i]}Q(t) \). As before consider the rescaled metrics 
\[
g_i(s) = Q_ig(t_i + s/Q_i).
\]
Non-local collapsing gives the desired injectivity radius estimates required for the compactness theorem in \cite{Gianniotis}. Since \( |R_i|\) and \( |H_i|\) are both bounded by \(1\) we can apply said compactness theorem and obtain a subsequence that converges locally to a smooth flow \( (M_\infty, g_\infty(s), p_\infty)\).

Now we proceed to prove the theorem. 

\begin{enumerate}
    \item{Genuinely \(H\)-dominated case: }  In this case \(p_i\) belongs to the boundary, and \( Q_i = H_{\max}(t_i)^2=: h_i^2\). Thus \(M_\infty\) is flat and has nonempty convex boundary. Observe that \(M_\infty\) must be noncompact. Otherwise, the convergence is global and the total curvature in the limit would vanish, contradicting the fact that \( C_0 > 0 \). This proves part (1) in Theorem \ref{thm:nonsymmetric disc}.

    \item{ Bounded \(H^2/R\) case: } In this case we can choose \( Q_i = R_{\max}(t_i)\); hence, the limit will have nonnegative scalar curvature with \(R_{\infty}(p_{\infty}) = 1\) and, if the boundary remains it will be convex. Let us now preclude every noncompact limit. Let \( K \subseteq M_{\infty}\) be any compact subset and \(K_i\) be its image under the Cheeger--Gromov diffeomorphisms. Note that if the limit is noncompact, then \(A_i \rightarrow \infty\); thus, we have that 
    \[
    r_i = \frac{C_0}{A_i} \rightarrow 0.
    \]
    By the local convergence we obtain that 
    \begin{align*}
        \int_{s_1}^{s_2}\int_K R_{\infty}^2 \dmu_{\infty}\, ds &= \lim_{i\rightarrow \infty} \int_{s_1}^{s_2} \int_{K_i}\left( R_i - r_i \right)^2\dmu_i \, ds\\
        &\leq \lim_{i\rightarrow \infty} \int_{s_1}^{s_2}\int_M (R_i - r_i)^2\dmu_i\, ds\\
        &= \lim_{i\rightarrow \infty} \mathcal{F}(g(t_i + s_1/Q_i)) - \mathcal{F}(g(t_i+ s_2/Q_i)) = 0.
    \end{align*}
    Therefore, \( R_\infty \equiv 0\) on every compact subset, that is, \(M_{\infty}\) is flat. But this is impossible, since \(R_{\infty}(p_{\infty}) = 1\). Then the limit must be compact, and by Theorem \ref{thm:compact-blow-up} the result follows.
    \begin{remark}
        The computation above actually shows that any blow-up limit in the \(R\)-scale must be compact.
    \end{remark}
\end{enumerate}

\section{The non-positive total-curvature regime}
\label{sec:nonpositive-total-curvature}

We now prove the convergence result stated in
Theorem~\ref{thm:nonpositive-curvature}.  As explained in Remark~\ref{rmk: equivalence with Li}, the curvature Neumann condition considered by Li~\cite{Li93} is dynamically equivalent to preservation of the geodesic-curvature density.  Li proved the corresponding convergence theorem using Hamilton's curvature potential and maximum-principle
estimates for the curvature.  Here we give a proof adapted to the variational formulation of the flow.  We first obtain a canonical constant-curvature metric by minimising the normalised Liouville functional.  We then use this metric as a fixed background: maximum principle estimates  for the conformal factor give convergence when $C_{0}<0$, whereas a monotonic entropy for the logarithmic diffusion equation treats the case $C_{0}=0$.

\subsection{The constant-curvature reference metric}

Recall that
\[
\mathcal{F}(u)
=
\int_{M}\bigl(|\nabla_{0}u|^{2}+R_{0}u\bigr)\dmu_{0}
-\frac{C_{0}}{2}
 \log\left(\frac{1}{A_{0}}\int_{M}e^{2u}\dmu_{0}\right)
\]
is invariant under constant rescalings of the metric.  When $C_{0}\leq0$, its coercivity modulo constants follows directly from Jensen's inequality.

\begin{proposition}\label{prop:nonpositive-reference-metric}
Suppose that $C_{0}\leq0$ and set $r=C_{0}/A_{0}$.  There exists a unique
smooth function $u_{*}$ satisfying
\[
\int_{M}e^{2u_{*}}\dmu_{0}=A_{0},
\qquad
-2\Delta_{0}u_{*}+R_{0}=r e^{2u_{*}},
\qquad
\partial_{N_{0}}u_{*}=0.
\]
Consequently, the metric
\[
g_{*}=e^{2u_{*}}g_{0}
\]
is the unique metric in the conformal class of $g_{0}$ satisfying
\[
\Area(M,g_{*})=A_{0},
\qquad
R_{g_{*}}=r,
\qquad
H_{g_{*}}\dsigma_{g_{*}}=H_{0}\dsigma_{0}.
\]
\end{proposition}

\begin{proof}
Let us consider the space
\[
\mathcal H
=
\left\{
u\in H^{1,2}(M):\int_{M}u\dmu_{0}=0
\right\}.
\]
For $u\in\mathcal H$, Jensen's inequality then gives
\[
\frac{1}{A_{0}}\int_{M}e^{2u}\dmu_{0}
\geq
\exp\left(\frac{2}{A_{0}}\int_{M}u\dmu_{0}\right)
=1.
\]
Since $C_{0}\leq0$, the logarithmic term in $\mathcal{F}$ is therefore
nonnegative.  The Poincar\'e and Cauchy--Schwarz inequalities imply
\[
F(u)
\geq
\|\nabla_{0}u\|_{L^{2}}^{2}
-\|R_{0}\|_{L^{2}}\|u\|_{L^{2}}
\geq
\|\nabla_{0}u\|_{L^{2}}^{2}
-B\|\nabla_{0}u\|_{L^{2}}.
\]
Thus $\mathcal{F}$ is coercive on $\mathcal H$ and bounded from below. Moreover, it is not difficult to see that it is weakly lower semicontinuous; therefore, applying the direct method from calculus of variations gives minimiser \(\widehat{u}\) of the functional on $\mathcal{H}$. Additionally, as in the previous section, that minimiser will be smooth. Hence, we can choose a constant $c_*$ so that the areas with respect to $g_0$ and that to $g_{*} : = e^{2u_*}g_0$ (where $u_* = \widehat u + c_* $) are equal. Hence, since \(\mathcal{F}\) is invariant under constant rescaling we have that \(u_*\) is a smooth minimiser of the functional on \( H^{1,2}(M)\).

It remains to prove uniqueness.  Suppose $u_{1}$ and $u_{2}$ both satisfy
the displayed equations and put $q=u_{1}-u_{2}$.  Subtracting the two
equations, multiplying by $q$, and integrating by parts gives
\[
2\int_{M}|\nabla_{0}q|^{2}\dmu_{0}
=
r\int_{M}\bigl(e^{2u_{1}}-e^{2u_{2}}\bigr)q\dmu_{0}
\leq0.
\]
The inequality follows because $(e^{2u_1} - e^{2u_2})(u_1 - u_2)$ is always nonnegative while $r\leq 0$.  Thus $q$ is constant.  The common area normalisation forces that constant
to be zero, including when $r=0$.  The remaining geometric conclusions
follow from the conformal transformation laws.
\end{proof}

Write the normalised solution relative to this reference metric as
\[
\widetilde g(\widetilde t)
=e^{2v(\widetilde t)}g_{*}.
\]
In particular, $v(0)=-u_{*}$.  Since $R_{g_{*}}=r$, its scalar curvature is
\[
R_{\widetilde g}
=e^{-2v}\bigl(r-2\Delta_{*}v\bigr).
\]
It follows from $2\partial_{\widetilde t}v=r-R_{\widetilde g}$ that
\begin{equation}\label{eq:relative-normalised-flow}
\partial_{\widetilde t}v
=e^{-2v}\Delta_{*}v
+\frac r2\bigl(1-e^{-2v}\bigr),
\qquad
\partial_{N_{*}}v=0.
\end{equation}
Moreover, preservation of area gives
\begin{equation}\label{eq:relative-area-constraint}
\int_{M}e^{2v(\widetilde t)}\dmu_{*}
=A_{0}
=\int_{M}\dmu_{*}.
\end{equation}

We shall use the following standard regularity result repeatedly.
If $v$ stays bounded in $C^{0}$, then
\eqref{eq:relative-normalised-flow} is uniformly parabolic.  Uniform
boundary H\"older estimates on fixed length space-time cylinders $M\times [t_1,t_2] $ give a $C^{\alpha}$ bound; the
coefficient $e^{-2v}$ is then H\"older continuous, and the Schauder estimates can be iterated.  Since $M$ is compact and its boundary
is smooth, finitely many interior and boundary charts make the constants
uniform in time.  Thus, for every $k\geq0$ and every $\tau>0$,
\begin{equation}\label{eq:uniform-parabolic-estimates}
\sup_{\widetilde t\geq\tau}
\|v(\widetilde t)\|_{C^{k}(M,g_{*})}
\leq C_{k,\tau}.
\end{equation}

\subsection{Negative total scalar curvature}

Suppose first that $C_{0}<0$, so $r<0$.  Let
\[
M(\widetilde t)=\max_{M}v(\widetilde t),
\qquad
m(\widetilde t)=\min_{M}v(\widetilde t).
\]
The area constraint \eqref{eq:relative-area-constraint} implies
\[
m(\widetilde t)\leq0\leq M(\widetilde t),
\]
unless $v$ vanishes identically.  The Hopf maximum principle then gives
\[
M'
\leq
\frac r2\bigl(1-e^{-2M}\bigr),
\qquad
m'
\geq
\frac r2\bigl(1-e^{-2m}\bigr).
\]
In particular, $M$ is nonincreasing and $m$ is nondecreasing.  Hence $v$
remains between its initial extrema, so the equation is uniformly parabolic
and the solution exists for all $\widetilde t\geq0$. Furthermore, we obtain that there is a constant $c>0$ such that the following bound is true
\begin{equation}\label{eq:negative-exponential-c0}
\|v(\widetilde t)\|_{C^{0}}
\leq
C e^{-c\widetilde t}.
\end{equation}

The uniform estimates \eqref{eq:uniform-parabolic-estimates}, followed by
interpolation between $C^{0}$ and a sufficiently high $C^{j}$ norm, upgrade
\eqref{eq:negative-exponential-c0}: for every $k\geq0$ there are constants
$C_{k},c_{k}>0$ such that
\[
\|v(\widetilde t)\|_{C^{k}(M,g_{*})}
\leq
C_{k}e^{-c_{k}\widetilde t}.
\]
Consequently, $v$ is converging smoothly and exponentially to $0$, that is, $\widetilde g(\widetilde t)$ converges smoothly and
exponentially to $g_{*}$ when $C_{0}<0$.

\subsection{Zero total scalar curvature}

Assume now that $C_{0}=0$, so $r=0$ and $g_{*}$ is flat.  Equation
\eqref{eq:relative-normalised-flow} becomes
\[
\partial_{\widetilde t}v=e^{-2v}\Delta_{*}v,
\qquad
\partial_{N_{*}}v=0.
\]
Note also that in this case, the unnormalised flow and the normalised one coincide. The maximum principle this time only shows that 
\[
m(0) \leq v(\wt t) \leq M(0)
\]
Thus $v$ is uniformly bounded, the equation
is uniformly parabolic, and the solution again exists for all time.

Introduce the relative conformal factor
\[
\omega=e^{2v}.
\]
It satisfies the logarithmic diffusion equation
\begin{equation}\label{eq:logarithmic-diffusion}
\partial_{\widetilde t}\omega
=\Delta_{*}\log \omega,
\qquad
\partial_{N_{*}}\omega=0,
\qquad
\frac1{A_{0}}\int_{M}\omega\dmu_{*}=1.
\end{equation}
By the inequalities above, we obtain the corresponding inequalities
\[
0<w\leq \omega(x,\widetilde t)\leq W<\infty.
\]

Now let us consider the entropy functional
\[
\mathcal E(\omega)
=
\int_{M}\Phi(\omega)\dmu_{*},
\qquad
\Phi(s)=s\log s-s+1.
\]
The function $\Phi$ is nonnegative and vanishes only at $s=1$.  Since
$\Phi'(s)=\log s$ and the boundary condition in
\eqref{eq:logarithmic-diffusion} eliminates the boundary term, computing its variation yields
\begin{equation}\label{eq:logarithmic-entropy-dissipation}
\frac{d}{d\widetilde t}\mathcal E(\omega)
=
-\int_{M}|\nabla_{*}\log \omega|^{2}\dmu_{*}.
\end{equation}

On the interval $[w,W]$, Taylor's theorem and
$\Phi''(s)=1/s$ give
\begin{equation}\label{eq:entropy-quadratic-comparison}
\frac{1}{2W}(\omega-1)^{2}
\leq
\Phi(\omega)
\leq
\frac{1}{2w}(\omega-1)^{2}.
\end{equation}
Let $\lambda_{1}>0$ be the first nonzero Neumann eigenvalue of
$-\Delta_{*}$.  Since $\omega$ has mean \(1 \), the Neumann Poincar\'e inequality,
\eqref{eq:logarithmic-entropy-dissipation}, and
\eqref{eq:entropy-quadratic-comparison} imply
\[
\begin{split}
\frac{d}{d\widetilde t}\mathcal E(\omega)
&\leq
-\frac{1}{W^{2}}
 \int_{M}|\nabla_{*}\omega|^{2}\dmu_{*} \\
&\leq
-\frac{\lambda_{1}}{W^{2}}
 \int_{M}(\omega-1)^{2}\dmu_{*} \\
&\leq
-\frac{2w\lambda_{1}}{W^{2}}
 \mathcal E(\omega).
\end{split}
\]
Therefore, with
\[
\eta=\frac{2w\lambda_{1}}{W^{2}}>0,
\]
after integration we have
\[
\mathcal E(\omega(\widetilde t))
\leq
\mathcal E(\omega(0))e^{-\eta\widetilde t},
\]
and using that $\| \omega - 1\|^2_{L^2} \leq 2W\mathcal{E}(\omega)$, we can conclude that $\omega \rightarrow 1$ exponentially in $L^2(\dmu_*)$.

For completeness, the uniform smooth estimates used to pass from this
$L^{2}$ decay to smooth convergence follow directly from the uniform
parabolicity given by the bounds $w\leq \omega\leq W$.
Boundary H\"older estimates and Schauder bootstrapping therefore yield,
for all $k\geq0$,
\[
\sup_{\widetilde t\geq1}
\|\omega(\widetilde t)\|_{C^{k}(M,g_{*})}<\infty.
\]
Hence, the family \(\{\omega(\wt t): \wt t\geq 0 \}\) is precompact in $C^\infty(M)$, but every subsequence must converge to \(1\). So that \(\omega \rightarrow 1\) in \( C^\infty\) as \(\wt t \rightarrow \infty\). It follows that 
$\widetilde g(\widetilde t)\to g_{*}$, smoothly.
To obtain exponential convergence in every \( C^k\) proceed as follows: first note that the \( C^k\) bounds also give uniform bounds in the Sobolev spaces \( H^{k,2}\). Then, the result follows after applying the Sobolev embedding theorem and interpolation with \( L^2\) and \(H^{j,2}\) with \( j\) sufficiently large. This
proves Theorem~\ref{thm:nonpositive-curvature}.

\section{Logarithmic diffusion equation}
Throughout this section we restate our results in terms of the logarithmic diffusion equation on surfaces. We begin by deriving the equation on the interior and the corresponding boundary condition induced by our free boundary flow. To this end, let \((M,g_0)\) be a compact surface with boundary and let \(g(t)\) be the solution to the Ricci flow \eqref{eq:free-boundary-flow-intro}. Writing \( g(t) = \omega(t)g_0\) we see that relative conformal factor satisfies the equation 
\[
\p_t \omega = - R\, \omega.
\]
Since the scalar curvature under conformal change transforms by \( R = \omega^{-1} ( R_0 - \Delta_0 \log \omega) \), we find that \( \omega\) evolves by 
\[
\p_t \omega = \Delta_0\log \omega - R_0.
\]
Moreover, under our flow we also have \( \p_{N_0}\omega = 0\). So that, with respect to the initial metric \( g_0\), the conformal factor \(\omega\) satisfies a logarithmic diffusion equation with force term \( R_0 \) and Neumann boundary conditions. However, in the case of the disc with \( C_0 > 0 \), it is convenient to consider the conformal factor with respect to the euclidean metric \(g_E\) while in the case of general surfaces with \( C_0 < 0\) with respect to the metric of constant scalar curvature \(g_*\) constructed in the previous section.\\

In the first case, we have that \( g_0 = \omega_0 g_E\) for some smooth function \( \omega_0 \). Writing \( g(t) = \omega(t) g_E\) we see that 
\begin{equation}
    \p_t \omega = \Delta_E \log \omega,
\end{equation}
so that on the interior we no longer have a force term. Nevertheless, for the boundary condition we obtain 
\begin{equation}
    \p_{N_E} \log \omega =  \p_{N_E} \log \omega_0.
\end{equation}
In the rotationally symmetric case, \(  -\p_{N_E} \log \omega_0\) is a constant which we denote \( \beta \). Notice that the total scalar curvature \(C_0\) can be expressed in terms of \(\beta\) as follows.
\[
    \int_M R\dmu =  - \int_M\Delta_E \log \omega \, dx = - \int_{\mathbb{S}^1} \p_{N_E}\log \omega \, d\theta = 2\pi \beta.
\]
Translating the result given by Theorem \ref{thm:rotational-disc} we obtain the following.
\begin{corollary}
    Let \(\omega : M \rightarrow \mathbb{R}\) be a positive rotationally symmetric solution to the logarithmic diffusion equation
    \begin{equation}
        \begin{cases}
            \p_t \omega = \Delta_E\log \omega,\\
            \p_{N_E}\log \omega = -\beta,
        \end{cases}
    \end{equation}
    with initial condition \( \omega(0) = \omega_0 \). If \(0< \beta < 2\) and it is constant, then  the solution ceases to exist in finite time 
    \[
    T = \frac{A_0}{2\pi \beta},
    \]
    where \( A_0 = \int_M \omega_0 \, dx \). Moreover, the function 
    \[
        \wt \omega (t) = \frac{A_0\, \omega(t)}{2\pi\beta(T - t)}
    \]
    converges in \( C^\infty \) as \( t \rightarrow T\) to the function \( \omega_\mathrm{cap}\) such that 
    \[
    g_{\mathrm{cap}} = \omega_{\mathrm{cap}}g_E.
    \]
\end{corollary}

Similarly, in the second case we have that \( g_0 = \omega_0 g_*\). This time we let \( \wt g(t) = \omega(t) g_*\) be the solution to the normalised free boundary flow. Then on the interior we have 
\[
\p_t \omega =  \left(r - \wt R\right)\omega = \Delta_*\log \omega - R_* + r\omega = \Delta_* \log \omega + r(\omega -  1),
\]
where the last equality follows from the fact that the scalar curvature \( R_*\) is constant. Additionally, \( \omega\) satisfies the Neumann condition 
\[
\p_{N_*}\omega = 0.
\]
Notice that \( \omega = 1\) is an equilibrium point on the equation. We have then the following result.
\begin{corollary}
    Let \((M, g_*)\) be a compact surface with boundary such that the metric \(g_*\) has constant scalar curvature. Then the solution \( \omega\) to the logarithmic diffusion equation
    \[
    \begin{cases}
        \p_t\omega &= \Delta_* \log \omega + r(\omega - 1)\\
        \p_{N_*}\omega &= 0.
    \end{cases} 
    \]
    with initial condition \(\omega(0) = \omega_0\) exists for all time and converges exponentially in \(C^\infty\) to the equilibrium point \(\omega_* = 1\) provided \( r \leq 0\) and the compatibility condition 
    \[
    \Area_{g_*}(M) = \int_M\omega_0 \dmu_*
    \]
\end{corollary}

\section*{Statements and Declarations}

\paragraph*{Funding.}
None

\paragraph*{Competing interests.}
None

\paragraph*{Data availability.}
No datasets were generated or analysed during this study.

\bibliographystyle{plainnat}
\bibliography{ref}

\end{document}